\documentclass[12pt,reqno]{amsart}
\usepackage{latexsym,amssymb,amscd}
\def\kk{{\Bbbk}}
\def\B'c{{\mathcal{B'}}}
\def\U'c{{\mathcal{U'}}}

\def\opn#1#2{\def#1{\operatorname{#2}}} % to make operators
\opn\chara{char}
\opn\length{\ell}
\opn\projdim{proj\,dim}
\opn\injdim{inj\,dim}
\opn\ini{in}
\opn\rank{rank}
\opn\Tiefe{Tiefe}
\opn\grade{grade}
\opn\height{height}
\opn\embdim{emb\,dim}
\opn\codim{codim}

\opn\Tr{Tr}
\opn\bigrank{big\,rank}
\opn\superheight{superheight}\opn\lcm{lcm}
\opn\trdeg{tr\,deg}%
\opn\reg{reg}
\opn\lreg{lreg}
\opn\deg{deg}
\opn\lcm{lcm}

\opn\div{div}
\opn\Div{Div}
\opn\cl{cl}
\opn\Cl{Cl}
\opn\Spec{Spec}
\opn\Supp{Supp}
\opn\supp{supp}
\opn\Sing{Sing}
\opn\Ass{Ass}
\opn\Min{Min}
\opn\Ann{Ann}
\opn\Rad{Rad}
\opn\Soc{Soc}
\opn\Ker{Ker}
\opn\Coker{Coker}
\opn\Im{Im}
\opn\Hom{Hom}
\opn\Tor{Tor}
\opn\Ext{Ext}
\opn\End{End}
\opn\Aut{Aut}
\opn\id{id}

\opn\nat{nat}
\opn\GL{GL}
\opn\SL{SL}
\opn\mod{mod}
\opn\ord{ord}
\opn\depth{depth}
\opn\set{set}
\opn\Shad{Shad}
\opn\aff{aff}
\opn\con{conv}
\opn\relint{relint}
\opn\st{st}
\opn\lk{lk}
\opn\cn{cn}
\opn\core{core}
\opn\vol{vol}
\opn\gr{gr}

\def\pot#1#2{#1[\kern-0.28ex[#2]\kern-0.28ex]}

\opn\dirlim{\underrightarrow{\lim}}
\opn\invlim{\underleftarrow{\lim}}
\def\pnt{{\raise0.5mm\hbox{\large\bf.}}}

\def\twoline#1#2{\aoverb{\scriptstyle {#1}}{\scriptstyle {#2}}}
\newcommand{\aoverb}[2]{{\genfrac{}{}{0pt}{1}{#1}{#2}}}
\def\Implies{\ifmmode\Longrightarrow \else
     \unskip${}\Longrightarrow{}$\ignorespaces\fi}
\def\implies{\ifmmode\Rightarrow \else
     \unskip${}\Rightarrow{}$\ignorespaces\fi}
\def\iff{\ifmmode\Longleftrightarrow \else
     \unskip${}\Longleftrightarrow{}$\ignorespaces\fi}

\let\:=\colon
\newtheorem{Theorem}{Theorem}[section]
\newtheorem{Lemma}[Theorem]{Lemma}
\newtheorem{Corollary}[Theorem]{Corollary}
\newtheorem{Proposition}[Theorem]{Proposition}

\let\epsilon=\varepsilon
\let\phi=\varphi
\let\kappa=\varkappa
\title{Normally torsion-free lexsegment ideals}
\author{Anda Olteanu}
\address{Faculty of Mathematics and Computer Science, Ovidius University, Bd.\ Mamaia 124,
 900527 Constanta, Romania,} \email{olteanuandageorgiana@gmail.com}
 \thanks{The author was supported by the BitDefender Postdoctoral Fellowship at the ``Simion Stoilow" Institute of Mathematics of the Romanian Academy (IMAR) and by the CNCSIS-UEFISCSU project PN II-RU PD 23/2010}

\begin{document}

\maketitle

\begin{abstract}  In this paper we characterize all the lexsegment ideals which are normally torsion-free. Our characterization is given in terms of the ends of the lexsegment. We also prove that the property of being normally torsion-free is equivalent to the property of the depth function of being constant.


\end{abstract}

\section*{Introduction} Powers of arbitrary ideals in general, and of monomial ideals in particular have been intensively studied during last years. Conca's examples \cite{C} of ideals with linear quotients generated in one degree whose powers do not have a linear resolution show that the powers of ideals do not generally preserve the homological properties and the invariants of the ideal. It is interesting to see how much these invariants may vary compared to those of the given ideal. 

Let $S=\kk[x_1,\ldots,x_n]$ be the polynomial ring in $n$ variables over a field $\kk$ and $I\subset S$ be a graded ideal. A well-known result of Brodmann \cite{B} says that the sets of associated primes of powers of $I$, $\Ass(S/I^t)$, stabilize for large $t$. An interesting case is that when $t=1$. In this case the ideal is called normally torsion-free. More precisely, an ideal $I$ is \textit{normally torsion-free} if $\Ass(S/I)=\Ass(S/I^k)$, for all $k\geq2$. Examples of normally torsion-free ideals appear from graph theory. It is known that a graph $G$ is bipartite if and only if its edge ideal is normally torsion-free, \cite{SVV}. Even if squarefree monomial ideals which are normally torsion-free have been intensively studied \cite{S}, \cite{SVV}, \cite{TM}, normally torsion-free monomial ideals which are not squarefree are almost unknown. We aim at characterizing all lexsegment ideals which are normally torsion-free. This will provide a large class of normally torsion-free monomial ideals which are not squarefree. We recall that, if $d\geq2$ is an integer and $u$ and $v$ are two monomials of degree $d$ in $S$ such that $u\geq_{lex}v$, then the monomial ideal generated by all the monomials $m$ of degree $d$ such that $u\geq_{lex}m\geq_{lex}v$ is called a lexsegment ideals. Lexsegment ideals were defined by Hullet and Martin \cite{HM} and they were also studied by Aramova, De Negri and Herzog \cite{ADH}, \cite{DH}. 

The paper is organized in two sections. Since there is a strong connection between the property of an ideal $I\subset S$ of having the maximal graded ideal as an associated prime ideal and the depth, more precisely, for an ideal $I$ of $S$, $\depth(S/I)=0$ if and only if $\frak{m}=(x_1,\ldots,x_n)\in\Ass(S/I)$, the first section is devoted to the study of the depth of the powers of lexsegment ideals. We prove that, if $d\geq 2$ is an integer, $u\geq_{lex}v$ are two monomials of degree $d$ which do not satisfy any of the following conditions: (i) $x_2\nmid u$, $x_2^d\geq_{lex} v>_{lex}x_{2}^{d-1}x_{\min(u/x_1)}$ and $w>x_2u/x_1$, where $w$ is the greatest monomial of degree $d$ such that $w<_{lex}v$, (ii) $d=2$, $u\leq_{lex}x_1x_3$ and $v=x_2^2$, and $I$ is the corresponding lexsegment ideal, then there exists some $k\geq1$ such that $\depth(S/I^k)=0$. Moreover, if $\depth(S/I^k)=0$ for some $k\geq1$, then $\depth(S/I^{k+j})=0$, for all $j\geq1$.

In the second section, we determine classes of lexsegment ideals which are normally torsion free. Using the results from the first section, one may easy see that these are the only classes of lexsegment ideals. For lexsegment ideals, properties such as having a linear resolution or being Cohen--Macaulay can be determined just by looking at the ends of the lexsegment \cite{ADH}, \cite{DH}. Our characterization for normally torsion-free lexsegment ideals is also given in terms of the ends of the lexsegment. Moreover, one may easy see that the property of being normally torsion-free is equivalent with the property of the depth of the powers of being constant.

\section{Depth of powers of lexsegment ideals}

In this section, we will focus on the depth of powers of lexsegment ideals. Through this paper, we denote by $S=\kk[x_1,\ldots,x_n]$ the polynomial ring in $n$ variables over a field $\kk$. We aim at finding all the lexsegment ideals $I$ which have the property that there exists some $k\geq1$ such that $\depth(S/I^k)=0$. In order to do this, we will prove that there exists some $k\geq1$ such that $\frak{m}=(x_1,\ldots,x_n)\in\Ass(S/I^k)$. Firstly, we recall the most frequently used concepts.

Let $I\subseteq S$ be a monomial ideal. A prime ideal $\frak{p}$ is an associated prime ideal of $I$ if there exists a monomial $m$ in $S$ such that $\frak{p}= I: (m)$. If we denote by $\Min(S/I)$ the set of minimal prime ideals over $I$, then it is known that $\Min(S/I)\subseteq\Ass(S/I^k)$ for every $k\geq 1$.

Let $m=x_1^{a_1}\cdots x_n^{a_n}$ be a monomial. For any $i$, $1\leq i\leq n$, we denote $\nu_i(m)=a_i$. The set $\supp(m)=\{i\ :\ \nu_i(m)\geq1\}$ is called \textit{the support} of the monomial $m$. One may define $\min(m):=\min(\supp(m))$ and $\max(m):=\max(\supp(m))$. Also, for a monomial ideal $I\subset S$, we will denote by $G(I)$ the minimal monomial generating set of $I$.

Let $<_{lex}$ be the lexicographical order on $S$ with respect to $x_1>_{lex}\cdots>_{lex}x_n$. We recall that, for two monomials $m$ and $m'$, $m<_{lex}m'$ if $\deg(m)<\deg(m')$ or $\deg(m)=\deg(m')$ and there exists $1\leq s\leq n$ such that, for any $i<s$, $\nu_i(m)=\nu_i(m')$ and $\nu_s(m)<\nu_s(m')$.

Let $d\geq2$ be an integer. We denote by $\mathcal{M}_d$ the set of all the monomials in $S$ of degree $d$. For two monomials $u$ and $v$ in $\mathcal{M}_d$ such that $u\geq_{lex}v$, one considers the set $\mathcal{L}(u,v)=\{m: m\in\mathcal{M}_d,\ u\geq_{lex}m\geq_{lex}v\}$ which is called a \textit{lexsegment set}. If $u=x_1^d$, then $\mathcal{L}^i(v):=\mathcal{L}(x_1^d,v)$ is called an \textit{initial lexsegment}, and if $v=x_n^d$ then $\mathcal{L}^f(u):=\mathcal{L}(u,x_n^d)$ is called a \textit{final lexsegment}. A \textit{(initial, final) lexsegment ideal} is a monomial ideal generated by a (initial, final) lexsegment set. 

Let $u,v\in \mathcal{M}_d$ be two monomials, $u\geq_{lex}v$, and $I=(\mathcal{L}(u,v))$ be the corresponding lexsegment ideal.
We note that we may always assume that $x_1\mid u$ and $x_1\nmid v$. Indeed, if $x_1\mid v$ we denote $u=x_1^{a_1}\cdots x_n^{a_n}$ and $v=x_1^{b_1}\cdots x_n^{b_n}$, with $a_1\geq b_1>0$. If $a_1=b_1$, then $I=(\mathcal{L}(u,v))$ is isomorphic, as an $S$-module, to the ideal generated by the lexsegment $\mathcal{L}(u/x_1^{a_1},v/x_1^{b_1})$ of degree $d-a_1$. This lexsegment may be studied in the polynomial ring in a smaller number of variables. If $a_1>b_1$, then $I=(\mathcal{L}(u,v))$ and $(\mathcal{L}(u/x_1^{b_1},v/x_1^{b_1}))$ are isomorphic as $S$-modules and we have $\nu_1(u/x_1^{b_1})>1$ and $\nu_1(v/x_1^{b_1})=0$.  Therefore, through this section, we will always assume that $x_1\mid u$ and $x_1\nmid v$. We denote by $w$ the largest monomial of degree $d$ such that $w<_{lex}v$.

The following result describes the depth for the powers of lexsegment ideals $I$ such that $\depth(S/I)=0$.
\begin{Proposition}\label{depthlex0} Let $u$ and $v$ be two monomials of degree $d$, $u\geq_{lex}v$, and $I=(\mathcal{L}(u,v))$. If $\depth(S/I)=0$, then $\depth(S/I^k)=0$ for any $k\geq1$. 
\end{Proposition}

\begin{proof} According to \cite[Proposition 3.2]{EOS}, $\depth(S/I)=0$ if and only if $x_nu\geq_{lex}x_1v$. 

Let us fix an arbitrary integer $k\geq 2$.
Firstly, let us assume that $\nu_1(u)>1$. In this case, the inequality $x_nu\geq_{lex}x_1v$ is obviously fulfilled. Let $m=u^k/x_1$, with $\deg(m)=kd-1$. Thus $m\notin I^k$. Since $x_im=(x_iu/x_1)u^{k-1}\in I^k$ for any $1\leq i\leq n$, we get that $\frak{m}\subseteq I^k:(m)$. The other inclusion being trivial, we get $\depth(S/I^k)=0$.

Let now $\nu_1(u)=1$, thus $u\leq_{lex}x_1x_2^{d-1}$. Since $x_nu\geq_{lex}x_1v$, we must have $v\leq_{lex}x_2^{d-1}x_n$. Let $m=u(x_2^{d-1}x_n)^{k-1}/x_1$ with $\deg(m)=kd-1$, thus $m\notin I^k$. Since $x_1m=u(x_2^{d-1}x_n)^{k-1}\in I^k$ and $x_im=(x_nu/x_1)(x_2^{d-1}x_i)(x_2^{d-1}x_n)^{k-2}\in I^k$, we get that $\frak{m}\subseteq I^k:(m)$. The other inclusion being obviously true, we get $\depth(S/I^k)=0$. 
\end{proof}

The proofs of the next results work as follows. In order to show that $\depth(S/I^k)=0$ we provide a monomial $m\in S$ of degree $dk-1$, thus $m\notin I^k$, such that $\frak{m}\subseteq I:(m)$. Since the other inclusion $I:(m)\subseteq\frak{m}$ is obviously true, we get that $\frak{m}\in\Ass(S/I^k)$, hence $\depth(S/I^k)=0$. 
Therefore, in the following proofs, we only show which is a right choice for the monomial $m$ in each case.

Firstly, for the case when $d=2$, we determine all the lexsegment ideals $I$ such that $\depth(S/I^k)=0$ for some $k\geq1$.

\begin{Proposition}\label{compd=2} Let $u,v\in\mathcal{M}_2$, $u\geq_{lex}v$, and $I=(\mathcal{L}(u,v))$ be the corresponding lexsegment ideal. 
\begin{itemize}
	\item[(a)] If $u\geq_{lex}x_1x_2$, then $\depth(S/I^{k})=0$, for any $k\geq2$. 
	\item[(b)] If $u=x_1x_3$, $x_2^2>_{lex}v\geq_{lex} x_2x_{n-1}$, then $\depth (S/I^k)=0$ for any $k\geq3$.
	\item[(c)] If $u<_{lex}x_1x_3$, $x_2x_{\max(u)}\geq_{lex}v\geq_{lex} x_2x_{n-1}$, then $\depth (S/I^k)=0$ for any $k\geq3$.
	\item[(d)] If $u\leq_{lex}x_1x_3$ and $v\leq_{lex}x_2x_n$ then $\depth (S/I^k)=0$ for any $k\geq2$.
\end{itemize}
\end{Proposition}

\begin{proof}  We now show which is a right choice for the monomial $m$ of degree $2k-1$ such that $\frak{m}\subseteq I^k:(m)$ for some $k$ in each case.

 (a) Let $m=x_1(x_2^{2})^{k-1}$. Then $x_1m=(x_1x_2)^2(x_2^{2})^{k-2}\in I^k$ and, for $2\leq i\leq n$, $x_im=(x_1x_i)(x_2^{2})^{k-1}\in I^k$. 

(b) Let $m=x_2^{2}u^{k-1}/x_1$. Then $x_1m=x_2^{2}u^{k-1}\in I^k$, $x_2m=(x_2x_3)(x_2^2)u^{k-2}\in I^k$ and $x_im=(x_2x_3)^2(x_1x_i)u^{k-3}\in I^k$ for every $3\leq i\leq n$. 

(c) Let $m=x_2^{2}u^{k-1}/x_1$. Then one has that the monomials $x_1m=x_2^{2}u^{k-1}$ and $x_2m=(x_2x_j)(x_2^2)u^{k-2}$ are in $I^k$. For $i\leq j$, one may write $x_im=(x_2x_i)(x_2x_j)u^{k-2}\in I^k$, and if $i>j$, $x_im=(x_2x_j)^2(x_1x_i)u^{k-3}\in I^k$. 

(d) Let $m=(x_2^{2})^{k-1}x_n$. Then we have that $x_1m=(x_2^{2})^{k-1}(x_1x_n)\in I^k$ and $x_im=(x_2^2)^{k-2}(x_2x_n)(x_2x_i)\in I^k$ for any $2\leq i\leq n$. 
\end{proof}

By using Corollary~\ref{d=2norm} and Corollary~\ref{depthnoncomp}, one may note that these are the only lexsegment ideals $I$ of degree $2$ such that $\depth(S/I^k)=0$ for some $k$.

Next, we will consider $d\geq3$ and $w<_{lex}x_2u/x_1$. Let $w$ be the largest monomial of degree $d$ such that $w<_{lex}v$. We are going to treat each of the following cases: $w<_{lex}x_2u/x_1$, $w=x_2u/x_1$, and $w>_{lex}x_2u/x_1$. 

In the next proposition we take $w<_{lex}x_2u/x_1$, thus $v\leq_{lex}x_2u/x_1$.

\begin{Proposition}\label{d>=3} Let $d\geq 3$ and $I=(\mathcal{L}(u,v))$ be a lexsegment ideal such that $v\leq_{lex}x_2u/x_1$. Then $\depth(S/I^{k})=0$, for all $k\geq2$. 
\end{Proposition}
\begin{proof} Let us fix $k\geq 2$. We only indicate a right monomial $m$ of degree $dk-1$ such that $\frak{m}\subseteq I^k:(m)$ for some $k$.

If $\nu_1(u)>1$, then $\depth(S/I)=0$ and the statement follows from Proposition~\ref{depthlex0}. Therefore, we may assume that $\nu_1(u)=1$.

Let $m=x_2^du^{k-1}/x_1$. Then $x_1m=u^{k-1}x_2^{d}\in I^k$ and, for any $2\leq i\leq n$, $x_im=(x_2u/x_1)(x_2^{d-1}x_i)u^{k-2}\in I^k$. 
If $x_2^{d-1}x_n\geq_{lex}v$, then $x_im\in I^k$ for all $i$. 

Let us assume that there exists $2\leq i\leq n$ such that $x_2^{d-1}x_i<_{lex}v$, that is $v=x_2^{d-1}x_j$, for some $j<i$. Since $v\leq_{lex}x_2u/x_1$, the monomial $u$ must be of the form $u=x_1x_2^{d-2}x_t$, with $t\leq j<i$. Therefore $x_1x_2^{d-2}x_i<_{lex}u$. In this case, we take $m'=(x_1x_2x_n^{d-2})^{k-1}u/x_1$ for which we get that the monomials $x_1m'=(x_1x_2x_n^{d-2})^{k-1}u$ and $x_im'=(x_1x_ix_n^{d-2})(x_1x_2x_n^{d-2})^{k-2}(x_2u/x_1)$ are in $ I^k$, for any $2\leq i\leq n$.  
\end{proof}

We consider next the case when $w=x_2u/x_1$, where, as above, $w$ is the largest monomial of degree $d$ such that $w<_{lex}v$.

\begin{Proposition}\label{d<=4} Let $d\geq3$ and $I=(\mathcal{L}(u,v))$ be a lexsegment ideal generated in degree $d$ such that $w=x_2u/x_1$. Then $\depth(S/I^k)=0$ for any $k\geq d$.
\end{Proposition}
\begin{proof} Let us fix $k\geq2$. As before, we find a right monomial $m$ of degree $dk-1$ such that $\frak{m}\subseteq I^k:(m)$ for some $k$.

 If $v>_{lex}x_2^{d-1}x_n$, we consider the monomial $m=x_1x_n^{d-2}(x_2^d)^{k-1}$. Then $x_1m=(x_1x_2^{d-2}x_n)(x_1x_2^2x_n^{d-3})(x_2^d)^{k-2}\in I^k,$ since $w\geq_{lex} x_2^{d-1}x_n$ implies $u\geq_{lex}x_1x_2^{d-2}x_n$. We also have $x_im=(x_1x_ix_{n}^{d-2})(x_2^{d})^{k-1}\in I^k$ for any $i\geq2$.

If $v\leq_{lex}x_2^{d-1}x_n$, let us choose $m=(x_2^{d-1}x_n)^{k}/x_n$. One may note that $x_1m=(x_1x_n^{d-1})(x_2^d)^{d-1}(x_2^{d-1}x_n)^{k-d}\in I^k$ and $x_im=(x_2^{d-1}x_i)(x_2^{d-1}x_n)^{k-1}\in I^k$ for any $i\geq2$.
\end{proof}

In the sequel, we study the depth in the case when $w>_{lex}x_2u/x_1$. In particular $\nu_1(u)=1$, therefore, if we denote $M=\min(u/x_1)$, we get $M\geq2$.

\begin{Proposition} Let $d\geq 3$ be an integer and $I=(\mathcal{L}(u,v))$ be a lexsegment ideal generated in degree $d$ such that $w>_{lex}x_2u/x_1$.
\begin{itemize}
	\item[(a)]  If $v\leq_{lex}x_2^{d-1}x_n$, then $\depth(S/I^k)=0$ for any $k\geq d$.
	\item[(b)]  If $x_2^{d-1}x_n<_{lex}v\leq_{lex}x_2^{d-1}x_M$, then $\depth(S/I^k)=0$ for any $k\geq2d$.
	\item[(c)]  If $x_2\mid u$ and $v=x_2^{d}$, then $\depth(S/I^k)=0$ for any $k\geq d$.
\end{itemize}
\end{Proposition}

\begin{proof}  In each case, we choose a right monomial $m$ of degree $dk-1$ such that $\frak{m}\subseteq I^k:(m)$ for some $k$.

(a) Let $k\geq d$ and $m=x_n^{d-1}(x_2^{d})^{k-1}$. Then $x_1m=(x_1x_n^{d-1})(x_2^{d})^{k-1}\in I^k$ and, taking into account that $k\geq d$, $x_im=(x_2^{d-1}x_n)^{d-1}(x_2^{d-1}x_i)(x_2^{d})^{k-d}\in I^k$, for all $2\leq i\leq n$.

(b) Let $k\geq2d$ and $m=(x_2^{d})^{k-d}u^d/x_1$ and assume $u=x_1x_M^{a_M}\cdots x_n^{a_n}$, thus $a_M\leq d-1$. Then $x_1m=(x_2^{d})^{k-d}u^d\in I^k$,  $$x_im=(x_2^{d-1}x_M)^{d-1}(x_2^{d-1}x_i)(x_2^{d})^{k-2d+1}u^{a_M+1}\prod_{\twoline{t\in\supp(u)}{t>M}}\left(\frac{x_tu}{x_M}\right)^{a_t}\in I^k,$$for all $2\leq i\leq M$, and, if $M<i\leq n$, $$x_im=(x_2^{d-1}x_M)^{d-1}\frac{x_iu}{x_M}(x_2^{d})^{k-2d+1}u^{a_M}\prod_{\twoline{t\in\supp(u)}{ t>M}}\left(\frac{x_tu}{x_M}\right)^{a_t}\in I^k.$$

(c) Let us fix $k\geq d$. Let $u=x_1x_M^{a_M}\cdots x_n^{a_n}$ and $m=u^k/x_1$. Then $x_1m\in I^k$ and, taking into account that $k\geq d$ and $x_2\mid u$,  $$x_im=(x_2^d)\frac{x_iu}{x_2}\prod_{j\in\supp(u/(x_1))}\left(\frac{x_ju}{x_2}\right)^{a_j}.$$ Since $x_ju/x_2\leq_{lex}u$, we get that $x_im\in I^k$.
\end{proof}

According to Corollary~\ref{d=2norm} and Corollary~\ref{depthnoncomp}, one may note that these are the only lexsegment ideals $I$ of degree $d\geq3$ such that $\depth(S/I^k)=0$ for some $k$.

Taking into account all the above results, we may conclude: \begin{Theorem}\label{depthpower0} Let $I=(\mathcal{L}(u,v))$ be a lexsegment ideal which does not satisfy any of the following conditions:
\begin{itemize}
 \item[(i)] $x_2^d\geq_{lex} v>_{lex}x_{2}^{d-1}x_{\min(u/x_1)}$ and $w>_{lex}x_2u/x_1$, where $w$ is the greatest monomial of degree $d$ with $w<_{lex}u$;
 \item[(ii)] $d=2$, $u\leq_{lex}x_1x_3$ and $v=x_2^2$.
 \end{itemize}
 Then there exists some $k\geq1$ such that $\depth(S/I^k)=0$. Moreover, if $\depth(S/I^k)=0$ for some $k\geq1$, then $\depth(S/I^{k+j})=0$, for any $j\geq1$.
\end{Theorem}

A particular class of lexsegment ideals for which we have a nice behavior for the depth of their powers is that of lexsegment ideals with a linear resolution.

\begin{Proposition} Let $I=(\mathcal{L}(u,v))$ be a lexsegment ideal with a linear resolution. Then $\depth(S/I^{k})=0$, for any $k\geq2$. 
\end{Proposition}
\begin{proof} Firstly, we consider that $I$ is a completely lexsegment ideal with a linear resolution. Using \cite[Theorem 1.3]{ADH}, we must have one of the following cases:

\begin{itemize}
	\item [(a)] $u=x_1^ax_2^{d-a},\ v=x_1^ax_n^{d-a}$ for some $a,\ 0< a\leq d;$
	\item [(b)] $b_1< a_1-1;$
	\item [(c)] $b_1 = a_1-1$ and $x_1 w/x_{\max(w)}\leq_{lex} u$, where $w$ is the largest monomial of degree $d$ such that $w <_{lex} v$. 
\end{itemize}
The case $d=2$ is obvious. Therefore, we assume $d\geq3$.
One may note that in cases (a) and (b) we have $\depth(S/I)=0$ according to \cite[Proposition 3.2]{EOS}. Therefore, by Proposition~\ref{depthlex0}, we have $\depth(S/I^{k})=0$, for any $k>1$.

For the case (c), we note that $w\leq_{lex}x_{\max(w)}u/x_1<_{lex}x_2u/x_1$ since $w<_{lex}{x_2^d}$ implies $\max(w)>2$. Therefore, $v\leq_{lex}x_2u/x_1$ and, by Proposition~\ref{d>=3}, $\depth(S/I^k)=0$, for any $k\geq2$.

Now, let us assume that $I$ is a lexsegment ideal with a linear resolution which is not a completely lexsegment ideal. We note that it is enough to prove that $\depth(S/I^2)=0$. Indeed, according to \cite{EO}, any power of a lexsegment ideal with a linear resolution has linear quotients. Therefore, by \cite[Proposition 2.1]{HH1}, the depth is a non-increasing function which implies $\depth(S/I^k)=0$, for any $k\geq2$.

In order to do this, we prove that $\frak{m}=(x_1,\ldots,x_n)\in\Ass(S/I^2)$. Since $I$ is a lexsegment ideal with a linear resolution which is not a completely lexsegment, according to \cite[Theorem 2.4]{ADH}, the monomials $u$ and $v$ must be of the form $u=x_1x_{l+1}^{a_{l+1}}\cdots x_n^{a_n},\ v=x_lx_n^{d-1}$ for some $2\leq l\leq n-1.$ Let $m=x_l^{d}u/x_1\notin I^2$. Then $x_1m=ux_l^{d}$ and $x_im=(x_l^{d-1}x_i)(x_lu/x_1)$ are in $I^2$, for all $2\leq i\leq n$. Therefore $\frak{m}\subseteq I^2:(m)$. Since the other inclusion is trivial, we get $\depth(S/I^2)=0$, which ends the proof.
\end{proof}

\section{Normally torsion-free lexsegment ideals}

We characterize all the lexsegment ideals which are normally torsion-free. The following result shows that we can reduce our study to those lexsegment ideals which have $x_1\mid u$ and $x_1\nmid v$.
\begin{Lemma} Let $u,v$ be two monomials of degree $d$ such that $x_1\mid u$ and $x_1\mid v$ and $I=(\mathcal{L}(u,v))$. Let $b=\nu_1(v)$ and we consider monomials $u'=u/x_1^b$ and $v'=v/x_1^b$. Let $I'=(\mathcal{L}(u',v'))$. Then $\Ass(S/I)=\{(x_1)\}\cup\Ass(S/I')$.
\end{Lemma}

\begin{proof}It is clear that $\{(x_1)\}\cup\Ass(S/I')\subseteq\Ass(S/I)$. Let us consider the converse inclusion. Let $\frak{p}\in\Ass(S/I)$. If $x_1\notin \frak{p}$, then we must have $\frak{p}\in\Ass(S/I')$. Therefore let $x_1\in\frak{p}$. If $\frak{p}=(x_1)$, then the statement is clear. We assume that there is some $i\geq2$ such that $x_i\in\frak{p}$. Since $\frak{p}\in\Ass(S/I)$, there exists a monomial $m\notin I$ such that $\frak{p}=I:(m)$. In particular, $x_im\in I$ and $x_1^b\mid m$. Let $m'=m/x_1^b$ and we prove that $\frak{p}=I':(m')$. Indeed, for any $x_j\in\frak{p}$ we have $x_jm\in I$. Thus, for any $x_j\in\frak{p}$, there exists $\alpha\in \mathcal{L}(u,v)$ such that $\alpha\mid x_jm$. This implies that $\alpha/x_1^b$ divides $x_jm'$ and $\alpha/x_1^b\in\mathcal{L}(u',v')$, thus $x_jm'\in I'$ and $\frak{p}\subseteq I':(m')$. For the converse inclusion, let $\beta \in I':(m')$. We have to prove that $\beta\in \frak{p}$. We have $\beta m'\in I'$, that is there exists a monomial $\alpha'\in\mathcal{L}(u',v')$ such that $\alpha'\mid\beta m'$. This implies $x_1^b\alpha'\mid\beta m$ and, since $x_1^b\alpha\in\mathcal{L}(u,v)$, $\beta\in I:(m)=\frak{p}$. 
\end{proof}

Henceforth, we will assume that $x_1\mid u$ and $x_1\nmid v$.
We will firstly consider the case when $\depth(S/(\mathcal{L}(u,v)))=0$. For this class of ideals, the set of associated prime ideals is known.

\begin{Proposition}[\cite{I}]\label{Ishaq} Let $I=(\mathcal{L}(u,v))$ be a lexsegment ideal with $\depth(S/I)=0$ which is not an initial ideal. Then $$\Ass(S/I)=\{(x_2,\ldots,x_n)\}\cup\{(x_1,\ldots,x_j):j\in\supp(v)\cup\{n\}\}.$$
\end{Proposition}

\begin{Proposition}\label{depth0} Let $I=(\mathcal{L}(u,v))$ be a lexsegment ideal with $\depth(S/I)=0$. Then $I$ is normally torsion-free.
\end{Proposition}

\begin{proof} Using \cite[Proposition 3.2]{EOS}, $\depth(S/I)=0$ is equivalent to the fact that $x_nu\geq_{lex}x_1v$. 
If $I=(\mathcal{L}^i(v))$ is an initial ideal such that $x_1\nmid v$, then $\Ass(S/I)=\{(x_1,\ldots,x_j):j\in\supp(v)\cup\{n\}\}$, \cite{HT}.

We have to show that $\Ass(S/I)=\Ass(S/I^k)$, for any $k\geq 2$.

``$\subseteq$" Let $k\geq 2$ and $\frak{p}\in\Ass(S/I)$. If $x_1\notin \frak{p}$, then $I$ is not an initial lexsegment ideal and $\frak{p}=(x_2,\ldots,x_n)$, according to Proposition~\ref{Ishaq}. By the proof of \cite[Proposition 3.1]{EOS}, $(x_2,\ldots,x_n)\in\Min(S/I)\subseteq\Ass(S/I^k)$. Let $x_1\in\frak{p}$. By Proposition~\ref{depthlex0}, the ideal $(x_1,\ldots,x_n)\in\Ass(S/I^k)$.

Let us assume now that $\frak{p}=(x_1,\ldots,x_j)$, with $j\in\supp(v)$. We may assume that $j<n$. Let $m=(x_1x_n^{d-1})^{k-1}v/x_j\notin I^k$, since $\deg(m)=kd-1$. Then $\frak{p}=I^k:(m)$. Indeed, $x_1m=(x_1v/x_j)(x_1x_n^{d-1})^{k-1}\in I^k$, since $x_1v\leq_{lex}x_nu<_{lex}x_ju$. If $2\leq i\leq j$, we get $x_im=(x_iv/x_j)(x_1x_n^{d-1})^{k-1}\in I^k$ since $x_2^d\geq_{lex} x_iv/x_j\geq_{lex}v$. Therefore $\frak{p}\subseteq I^k:(m)$. For the other inclusion, we assume by contradiction that there exists a monomial $m'\in I^k:(m)$ such that $m'\notin\frak p$, that is $\supp(m')\subseteq\{j+1,\ldots,n\}$. One may easy note that $m'\notin I^k$. Since $m'\in I^k:(m)$, there exists $\omega\in G\left(I^k\right)$ such that $\omega/\gcd(\omega,m)\mid m'$, thus $\supp(\omega/\gcd(\omega,m))\subseteq\{j+1,\ldots,n\}$ and $\nu_j(\omega/\gcd(\omega,m))\leq\nu_j(m)=\nu_j(v)-1$. Let $a=\nu_1(\omega)$. We obviously have $a\leq k-1$. One may note that $\omega\geq_{lex}(x_1x_n^{d-1})^{a}v^{k-a}=T$, therefore there exists $s$ such that, for any $i<s$, $\nu_i(\omega)=\nu_i(T)$, and $\nu_s(\omega)>\nu_s(T)$. Since $\supp(\omega/\gcd(\omega,m))\subseteq\{j+1,\ldots,n\}$, we must have $\supp(\omega/\gcd(\omega,T))\subseteq\{j+1,\ldots,n\}$ and $s\geq j+1$. Then $\nu_j(\omega)=\nu_j(v^{k-a})=(k-a)\nu_j(v)$, contradiction with $\nu_j(\omega)<\nu_j(v)-1$. Thus $\frak p=I^k:(m)$.

``$\supseteq$" Let us fix $k\geq2$. Using Proposition \ref{depthlex0}, $\frak{m}\in\Ass(S/I^k)$ and $\frak{m}\in\Ass(S/I)$.

Let $\frak{p}\in\Ass(S/I^k)$ and assume that $x_1\in \frak{p}$. Let $j:=\max\{i:x_i\in\frak{p}\}$. Since $\frak{p}\in\Ass(S/I^k)$, there exists a monomial $m\notin I^k$ such that $\frak{p}=I^k:(m)$.

Firstly, we prove that, for any $1<i<j$, $x_i\in\frak{p}$. Indeed, let $1<i<j$ be an integer. Since $x_jm\in I^k$, there exist $\alpha_1,\ldots,\alpha_{k}\in \mathcal{L}(u,v)$ and $\beta$ a monomial in $S$ such that $x_jm=\alpha_1\cdots\alpha_k\beta$. Since $m\notin I^k$, we must have $x_j\nmid\beta$. Let $1\leq t\leq k$ be such that $x_j\mid\alpha_t$. Therefore $m=\alpha_1\cdots(\alpha_t/x_j)\cdots\alpha_k\beta.$ One may note that $x_1\nmid\alpha_t$ since, otherwise, $x_n\alpha_t/x_j\in\mathcal{L}(u,v)$, which implies $x_nm\in I^k$ and $x_n\in\frak{p}$, contradiction. Thus $x_i\alpha_t/x_j>_{lex}v$ and $x_im\in I^k$, that is $x_i\in I^k:(m)=\frak{p}$.

Next we claim that $j\in\supp(v)$. Let us assume by contradiction that $j\notin\supp(v)$. We assumed that $\frak{p}\neq\frak{m}$ and this implies $j\neq n$. Since $x_jm\in I^k$, there exist monomials $\alpha_1,\ldots,\alpha_{k}\in \mathcal{L}(u,v)$ and $\beta$ a monomial in $S$ such that $x_jm=\alpha_1\cdots\alpha_k\beta$. Let $1\leq t\leq k$ be such that $x_j\mid\alpha_t$. Therefore $m=\alpha_1\cdots(\alpha_t/x_j)\cdots\alpha_k\beta.$ As before, we cannot have $x_1\mid \alpha_t$ since, in this case, we obtain also that $x_sm\in I^k$, for all $s>j$, that is $I^k:(m)=\frak{m}$. We consider the monomial $x_nm=\alpha_1\cdots(x_n\alpha_t/x_j)\cdots\alpha_k\beta.$ If $x_n\alpha_t/x_j\geq_{lex}v$, we get that $x_n\in \frak p$, that is $\frak p=\frak{m}$, contradiction. Therefore, we assume that $x_n\alpha_t/x_j<_{lex} v$, that is there exists $s\geq2$ such that, for any $2\leq i<s$, $\nu_i(v)=\nu_i(x_n\alpha_t/x_j)$ and $\nu_s(v)>\nu_s(x_n\alpha_t/x_j)$. On the other hand, $\alpha_t\geq_{lex}v$ implies that there exists $s'\geq2$ such that, for any $2\leq i<s'$, $\nu_i(v)=\nu_i(\alpha_t)$ and $\nu_{s'}(v)<\nu_{s'}(\alpha_t)$. Since $\nu_s(v)>\nu_s(x_n\alpha_t/x_j)$, $\nu_{s'}(v)<\nu_{s'}(\alpha_t)$, and $x_j\nmid v$ we get that $\nu_j(\alpha_t)=1$ and $s>j$. Let $b=\nu_s(v)$. Then $x_s^{b+1}m=\alpha_1\cdots(x_s^{b+1}\alpha_t/x_j)\cdots\alpha_k\beta$ and we obviously have $x_s^{b+1}\alpha_t/x_j\in I$, that is $x_s^{b+1}\in\frak{p}$ and $x_s\in\frak{p}$, contradiction. Therefore, if $x_1\in \frak{p}$, we must have $\frak{p}\in\{(x_1,\ldots,x_j):j\in\supp(v)\cup\{n\}\}$.

Let us assume that $x_1\notin\frak{p}$ and we prove that $\frak{p}=(x_2,\ldots,x_n)$. Indeed, let $2\leq i\leq\min(u/x_1)$. Taking into account that $x_nu\geq_{lex}x_1v$, we have $i\leq\min(v)$ and $x_i\in\frak{p}$ since $\frak{p}\supset I^k$ and $(x_i^{d})^{k}\in I^k$. If $i>\min(u/x_1)$, then $(x_1x_i^{(d-1)})^{k}\in I^k$, and, since $x_1\nmid \frak{p}$ and $\frak{p}\supseteq I^k$, we must have $x_i\in\frak{p}$. Therefore $\frak{p}=(x_2,\ldots,x_n)\in\Ass(S/I)$.  
\end{proof}

According to the results from the first section, we have two remaining cases to study: (i) $d=2$, $u=x_1x_3$ and $v=x_2^2$ and (ii) $x_2^{d-1}x_M<_{lex}v\leq x_2^d$ and $w>_{lex}x_2u/x_1$, where $w$ is the greatest monomial of degree $d$ such that $w<_{lex}v$ and $M=\min(u/x_1)$.

Firstly, we consider the case when $u\leq_{lex}x_1x_3$ and $v=x_2^2$. The following result describes the behavior of the set of associated prime ideals when passing to polynomial rings with a smaller number of variables.

\begin{Lemma}[\cite{TM}]\label{subring} Let $I$ be a monomial ideal in $S$ and let $x_i$ be a variable such that $x_i\nmid m$ for any $m\in G(I)$. Then there is a one-to-one correspondence between the sets $\Ass_S(S/I)$ and $\Ass_S(S/(I, x_i))$ given by $\frak{p}\in\Ass_S(S/I)$ if and only if
$(\frak{p},x_i)\in\Ass_S(S/(I,x_i))$.
\end{Lemma}

\begin{Proposition}\label{d=2v=x_2^2} Let $d=2$ and $I=(\mathcal{L}(u,v))$ be a lexsegment ideal such that $u\leq_{lex}x_1x_3$ and $v=x_2^2$. Then
$I$ is normally torsion-free.
\end{Proposition}
\begin{proof} According to Lemma~\ref{subring}, it is enough to show the statement for the case $u=x_1x_3$. Therefore $I=(x_1x_3,\ldots,x_1x_n,x_2^2)$. One may note that $$I=(x_1)\cap(x_3,\ldots,x_n)+(x_2^2)=(x_1,x_2^2)\cap(x_2^2,x_3,\ldots,x_n),$$ which is a standard primary decomposition of $I$, therefore $$\Ass(S/I)=\{(x_1,x_2),(x_2,x_3,\ldots,x_n)\}.$$ 

Since $\Ass(S/I)=\Min(S/I)\subseteq\Ass(S/I^k)$, we only have to prove that $\Ass(S/I)\supseteq\Ass(S/I^k)$ for any $k\geq 2$.

 Let $k\geq 2$ and $\frak{p}\in\Ass(S/I^k)$. One may easy note that we must have $x_2\in\frak{p}$ since $(x_2^{2})^{k}\in I^k$. Since $(x_2,\ldots,x_n)$ is obviously a minimal prime ideal over $I^k$, we have only to consider the case when $x_1\in \frak{p}$.

Therefore, let $x_1\in\frak{p}$. We assume by contradiction that there exists $3\leq i\leq n$ such that $x_i\in\frak{p}$. Since $\frak{p}\in\Ass(S/I^k)$, there exists a monomial $m\notin I^k$ such that $\frak{p}=I^k:(m)$. In particular, $\frak{p}m\subseteq I^k$. Therefore $x_1m\in I^k$ and $x_im\in I^k$. Since $x_im\in I^k$, we get that $x_1\mid m$. Let $m=x_1^{a_1}\cdots x_n^{a_n}$. Since $x_1m\in I^k$ and $m\notin I^k$ we must have $a_1+1\leq a_3+\cdots+a_n$. Indeed, if $a_1+1>a_3+\cdots+a_n$, then $a_1\geq a_3+\cdots+a_n$ and, taking into account that the monomials from $\mathcal{L}(u,v)$ which are divisible by $x_1$ are of the form $x_1x_j$, with $j\geq 3$, we get that $m\in I^k$, contradiction. On the other hand, $x_im\in I^k$ and $m\notin I^k$ imply $a_1\geq a_3+\cdots+a_n+1$ using a similar argument. Thus $a_1\geq a_1+2$, contradiction. Therefore, the only associated prime ideal of $I^k$ which contains $x_1$ is $(x_1,x_2)$.
\end{proof}

\begin{Corollary}\label{d=2norm} Let $d=2$ and $I=(\mathcal{L}(u,v))$ be a lexsegment ideal such that $u=x_1x_i$ for some $3\leq i\leq n$ and $v=x_2^2$. Then
$\depth(S/I^k)=i-2$, for any $k\geq 1$.
\end{Corollary}

\begin{proof} If $i=3$, then, according to Proposition \ref{d=2v=x_2^2}, $\depth(S/I^k)=1$ for any $k\geq 1$. If $4\leq i\leq n$, then $x_{3},\ldots,x_{i-1}$ is a regular sequence on $S/I^k$, for any $i\geq 4$. Then $\depth(S/I^k)=\depth(S/(I^{k},x_3,\ldots,x_{i-1}))+i-3$, for any $k\geq 1$. Let us consider the ideal $I'=I\cap S'\subseteq S'=\kk[x_1,x_2,x_i,\ldots,x_n]$. By the first part of the proof, $\depth(S'/I'^k)=1$ for any $k\geq1$. On the other hand $\depth(S'/I'^k)=\depth(S/(I^{k},x_3,\ldots,x_{i-1}))$, therefore $\depth(S/I^k)=i-2$.
\end{proof}

We now consider that $x_2^{d-1}x_M<_{lex}v\leq x_2^d$ and $w>_{lex}x_2u/x_1$, where $w$ is the greatest monomial of degree $d$ such that $w<_{lex}v$ and $M=\min(u/x_1)$. We need the following lemma.

\begin{Lemma}\label{lemma} Let $I$ be a monomial ideal in the polynomial ring in $n$ variables over a field $\kk$, $S=\kk[x_1,\ldots,x_t,y_1,\ldots,y_s]$ with $s+t=n$, such that $I=JS+KS$ where $J\subset S_1=\kk[x_1,\ldots,x_t]$ and $K\subset S_2=\kk[y_1,\ldots,y_s]$. Then $\frak{p}\in \Ass_{S}(S/I)$ if and only if $\frak{p}=\frak{p}_1S+\frak{p}_2S$, where $\frak{p}_1\in\Ass_{S_1}(S_1/J)$ and $\frak{p}_2\in\Ass_{S_2}(S_2/K)$.
\end{Lemma}
\begin{proof} Let $J=\bigcap\limits_{i=1}^s Q_i$ and $K=\bigcap\limits_{j=1}^r Q'_j$ be the two standard primary decomposition of the ideals $J$ and $K$ in the polynomial rings $S_1$ and $S_2$ respectively. This means that $Q_i$ and $Q'_j$ are irreducible monomial ideals and the intersection is irredundant. Then $I=JS+KS=\bigcap\limits_{i,j}(Q_i+Q'_j)S$ is the standard primary decomposition of $I$. One has $\frak{p}\in \Ass_{S}(S/I)$ if and only if there exists $1\leq i\leq s$ and $1\leq j\leq r$ such that $\frak{p}=\sqrt{Q_i+Q'_j}S=\sqrt{Q_i}S+\sqrt{Q'_j}S$. This is equivalent to $\frak{p}=\frak{p}_iS+\frak{p}_j S$ for some $i$ and $j$, with $1\leq i\leq s$ and $1\leq j\leq r$, where $\frak{p_i}=\sqrt{Q_i}\subset S_1$  and $\frak{p_j}=\sqrt{Q'_j}\subset S_2$. Therefore $\frak{p}=\frak{p}_iS+\frak{p}_j S$ for some $i,j$, $1\leq i\leq s$ and $1\leq j\leq r$, where $\frak{p_i}\in\Ass_{S_1}(S_/J)$ and $\frak{p_j}\in\Ass_{S_2}(S_2/K)$.
\end{proof}

\begin{Proposition}\label{x2xm} Let $I=(\mathcal{L}(u,v))$ be a lexsegment ideal such that $v>_{lex}x_{2}^{d-1}x_M$ and $w>_{lex}x_2u/x_1$, where $w$ is the greatest monomial of degree $d$ with $w<_{lex}u$ and $M=\min(u/x_1)$. Then $I$ is normally torsion-free.
\end{Proposition}
\begin{proof} Since $w>_{lex}x_2u/x_1$ we have that $\nu_1(u)=1$. Moreover, since $x_1\nmid v$, we must have $M\geq 3$, that is $u\leq_{lex}x_1x_3^{d-1}$. Using Lemma~\ref{subring}, it is enough to consider the case when $v=x_2^{d-1}x_{M-1}$. 

We have to prove that $\Ass_S(S/I)=\Ass_S(S/I^k)$, for any $k\geq 2$.

``$\subseteq$" Let $k\geq 2$ and $\frak{p}\in\Ass_S(S/I)$, that is there exists a monomial $m\notin I$ such that $I:(m)=\frak{p}$. We have to prove that $\frak{p}\in\Ass_S(S/I^k)$, that is there exists a monomial $\omega\notin I^k$ such that $\frak{p}=I^k:_S (\omega)$.

Let us denote $J=x_1(\mathcal{L}_{S_1}^f(u/x_1))\subset S_1'=\kk[x_1,x_M,\ldots,x_n]$, and $K=(\mathcal{L}^i_{S_2}(v))\subset S_2=\kk[x_2,\ldots,x_{M-1}]$. We have $I=JS+KS$. Using Lemma~\ref{lemma}, we get that there exist $\frak{p}_1\in\Ass_{S'_1}(S'_1/J)$ and $\frak{p}_2\in\Ass_{S_2}(S_2/K)$ such that $\frak{p}=\frak{p}_1S+\frak{p}_2S$. 

Taking into account that $J=(x_1)\cap(\mathcal{L}_{S_1}^f(u/x_1))$, the only associated prime ideal of $J$ which contains $x_1$ is $(x_1)$.
Since we may write $K=x_2^{d-1}(x_2,\ldots,x_{M-1})$ we get $\Ass_{S_2}(S_2/K)=\{(x_2),(x_2,\ldots,x_{M-1})\}$. Thus, if $x_1\in\frak{p}$, then $\frak{p}=(x_1,x_2)$ or $\frak{p}=(x_1,x_2,\ldots,x_{M-1})$. In the first case, we have that $(x_1,x_2)\in\Min(S/I)$, and in the second case, $I^k:_S(x_n^{d-1}x_2^{dk-1})=(x_1,x_2,\ldots,x_{M-1})$, thus both ideals are in $\Ass_S(S/I^k)$.

Let us assume now that $x_1\notin\frak{p}$. Since $\frak{p}_1\in\Ass_{S'_1}(S'_1/J)$ there exists a monomial $\alpha\in S'_1$ such that $J:_{S'_1}(\alpha)=\frak{p}_1$.
If $\frak{p}_2=(x_2)$, then one may consider the monomial $\omega=\alpha v^{k}/x_2\notin I^k$ such that $\frak{p}=I^k:_S(\omega)$. If $\frak{p}_2=(x_2,\ldots,x_{M-1})$, then $\frak{p}=I^k:_S(\alpha x_2^{dk-1})$, therefore $\frak{p}\in\Ass_S(S/I^k)$.

``$\supseteq$" Let $k\geq 2$ and $\frak{p}\in\Ass_S(S/I^k)$. We have to prove that $\frak{p}\in\Ass_S(S/I)$, which, by Lemma~\ref{lemma}, means to show that there exist $\frak{p}_1\subset S'_1$, $\frak{p}_1\in\Ass_{S'_1}(S'_1/J)$ and $\frak{p}_2\subset S_2$, $\frak{p}_2\in\Ass_{S_2}(S_2/K)$ such that $\frak{p}=\frak{p}_1S+\frak{p}_2S$. Since $\frak{p}\supset I^k$, we get $\frak{p}\supset J^k$ and $\frak{p}\supset K^k$, therefore there exist $\frak{p}_1\subset S'_1$ and $\frak{p}_2\subset S_2$ such that $\frak{p}=\frak{p}_1S+\frak{p}_2S$. We have to prove that $\frak{p}_1\in\Ass_{S'_1}(S'_1/J)$ and $\frak{p}_2\in\Ass_{S_2}(S_2/K)$.

Since $\frak{p}\in\Ass_S(S/I^k)$, there exists a monomial $\omega\notin I^k$ such that $\frak{p}=I^k:_S (\omega)$. Using that $\frak{p}=\frak{p}_1S+\frak{p}_2S$, we have $\supp(\omega)\cap\{1,M,\ldots,n\}\neq\emptyset$ and $\supp(\omega)\cap\{2,\ldots,M-1\}\neq\emptyset$, that is $\omega=\omega_1\omega_2$, with $\omega_1\in S'_1$ and $\omega_2\in S_2$. Moreover, $\frak{p}\omega\subseteq J^{k_1}K^{k_2}$ with $k_1,k_2\geq1$ and $k_1+k_2=k$. One may easy note that $\omega_1\notin J^{k_1}$ and $\omega_2\notin K^{k_2}$, otherwise we get a contradiction. 

Firstly, we prove that $\frak{p}_1\in\Ass_{S'_1}(S'_1/J)$. Let us assume that $x_1\in\frak{p}$. We show that, for any $i$, $M\leq i\leq n$, $x_i\notin\frak{p}$. Assume that there exists $M\leq i\leq n$ such that $x_i\in\frak{p}$. In particular, we have $x_1\omega_1\in J^{k_1}$ and $x_i\omega_1\in J^{k_1}$. Since $J^{k_1}=x_1^{k_1}(\mathcal{L}_{S_1}^f(u/x_1))^{k_1}$ and using that $x_1\omega_1\in J^{k_1}$ and $\omega_1\notin J^{k_1}$, we get $x_1^{k_1-1}\mid \omega_1$ and $x_1^{k_1}\nmid \omega_1$. On the other hand, $x_i\omega_1\in J^{k_1}$ implies $x_1^{k_1}\mid \omega_1$, contradiction. Thus $\frak{p}_1=(x_1)\in\Min(S'_1/J)\subseteq\Ass_{S'_1}(S'_1/J)$. 

If $x_1\notin \frak{p}$, then $x_1\notin\frak{p}_1$. Taking into account that $(x_1x_{i}^{d-1})^{k}\in J^{k}$, for any $M+1\leq i\leq n$, and $\frak{p}\supseteq J^{k}$, we get $(x_{M+1},\ldots,x_n)\subset \frak{p}$, therefore $(x_{M+1},\ldots,x_n)\subseteq \frak{p}_1$. If $x_M\in\frak{p}$, then $\frak{p}_1=(x_M,\ldots,x_n)$, and, since $\frak{p}_1=J:_{S'_1}(u/x_M)$ we get $\frak{p}_1\in\Ass_{S'_1}(S'_1/J)$. If $x_M\notin\frak{p}$, we have $\frak{p}_1=(x_{M+1},\ldots,x_n)\in\Min(S'_1/J)$. Therefore $\frak{p}_1\in\Ass_{S'_1}(S'_1/J)$.

We now prove that $\frak{p}_2\in\Ass_{S_2}(S_2/K)$. We obviously have $x_2\in\frak{p}$ since $(x_2^{d})^{k}\in I^k$, thus $x_2\in \frak{p}_2$. We show that, if there exists $x_{i}\in\frak{p}$, $2<i<M$, then, $(x_2,\ldots,x_{M-1})\subseteq\frak{p}$. Indeed, if $x_{i}\in \frak{p}$, we get $x_{i}\omega_1\omega_2\in J^{k_1}K^{k_2}$, therefore $x_{i}\omega_2\in K^{k_2}$, that is there exists $\alpha_1,\ldots,\alpha_{k_2}\in \mathcal{L}_{S_2}^i(v)$ and $\beta\in S_2$ such that $x_i\omega_2=\alpha_1\cdots\alpha_{k_2}\beta$. In particular, there exists $1\leq t\leq k_2$ such that $\alpha_t=x_2^{d-1}x_i$. We get that $x_j\omega_2\omega_1=\alpha_1\cdots (x_j\alpha_t/x_i)\cdots\alpha_{k_2}\beta \omega_1\in I^k$ for any $2\leq j\leq M-1$. We obtain that $(x_2,\ldots,x_{M-1})\subseteq\frak{p}$. Thus $\frak{p}_2=(x_2,\ldots,x_{M-1})$ and $\frak{p}_2\in\Ass_{S_2}(S_2/K)$ since $\frak{p}=K:_{S_2}(x_{2}^{d-1})$. If $x_i\notin\frak{p}$ for any $3\leq i\leq M-1$, then $\frak{p}_2=(x_2)\in \Min(S_2/K)\subseteq\Ass_{S_2}(S_2/K)$.
\end{proof}

\begin{Corollary}\label{depthnoncomp} Let $I=(\mathcal{L}(u,v))$ be a lexsegment ideal such that $v=x_{2}^{d-1}x_{\ell}>_{lex}x_{2}^{d-1}x_M$ and $w>_{lex}x_2u/x_1$, where $w$ is the greatest monomial of degree $d$ with $w<_{lex}u$. Then $\depth(S/I^k)=M-\ell$, for any $k\geq1$.
\end{Corollary}

\begin{proof} We may assume that $v=x_{2}^{d-1}x_{M-1}$ since, otherwise, $x_{\ell+1},\ldots,x_{M-1}$ are regular on $S/I^k$, for any $k\geq1$. Using Proposition \ref{x2xm} and since $\frak{m}\notin\Ass_{S}(S/I)$, we get that $\depth(S/I^k)=1$ for any $k\geq1$, since $(x_2,\ldots,x_n)\in\Ass_S(S/I)$. Indeed, one may easy note that $(x_2,\ldots,x_n)=I:(x_1x_n^{d-2}x_2^{d-1})$. If $v=x_2^{d-1}x_{\ell}$, then $\depth(S/I^k)=M-\ell$, for any $k\geq1$.
\end{proof}

Now we may completely characterize all the lexsegment ideals which are normally torsion-free. One may note that our characterization depends on the ends of the lexsegment.

\begin{Theorem} Let $I=(\mathcal{L}(u,v))$ be a lexsegment ideal. The following condition are equivalent.
\begin{itemize}
 \item[(a)] $I$ is normally torsion-free. 
 \item[(b)] One of the following conditions hold:
\begin{itemize}
	\item[(i)] $x_nu\geq_{lex}x_1v$;
	\item[(ii)] $x_2^{d-1}x_M<_{lex}v\leq x_2^d$ and $w>_{lex}x_2u/x_1$, where $w$ is the greatest monomial of degree $d$ such that $w<_{lex}v$ and $M=\min(u/x_1)$;
	\item[(iii)] $d=2$, $u=x_1x_3$ and $v=x_2^d$. 
\end{itemize}
\item[(c)] $\depth(S/I)=\depth(S/I^k)$ for any $k\geq2$.
\end{itemize}
\end{Theorem}
\begin{proof} ``$(b)\Rightarrow(a)$" If $I$ satisfies one of the above conditions, then $I$ is normally torsion-free by Proposition~\ref{d=2v=x_2^2}, Proposition~\ref{x2xm}, and Proposition~\ref{depth0}.

``$(a)\Rightarrow(b)$" If $I$ does not satisfy any of the three conditions, then we have $\depth(S/I)\geq1$. Thus $\frak{m}\notin\Ass(S/I)$. According to the above results, there exists some $k\geq2$ such that $\depth(S/I^k)=0$, that is $\frak{m}\in\Ass(S/I^k)$. Therefore, there exists some $k\geq 2$ such that $\Ass(S/I)\neq\Ass(S/I^k)$ and $I$ is not normally torsion-free.

``$(b)\Leftrightarrow(c)$" The statement follows by using Corollary \ref{depthnoncomp}, Corollary \ref{d=2norm}, Proposition \ref{depthlex0}, and Theorem \ref{depthpower0}.
\end{proof}

\end{document}